\documentclass[11pt,twoside]{article}

\usepackage[T1]{fontenc}
\usepackage[utf8]{inputenc}
\usepackage{lmodern}
\usepackage{microtype}

\usepackage{amsmath,amssymb} % already in sn-jnl, but safe

\newcommand{\Ocal}{\mathcal{O}}

\newcommand{\SpecGnC}[1]{\operatorname{Spec}^{\Gamma,\mathrm{nc}}(#1)}

\DeclareMathOperator{\ExtG}{Ext^{\Gamma}}
\newcommand{\RExtG}{\mathbf{R}\!\ExtG}

\usepackage[a4paper,margin=1in]{geometry}
\usepackage{amsmath,amssymb,amsfonts,amsthm,mathtools}
\usepackage{bbm}
\usepackage{mathrsfs}
\usepackage{bm}
\usepackage{stmaryrd}
\usepackage{longtable,booktabs}

\usepackage{tikz-cd}
\usetikzlibrary{arrows.meta,decorations.markings,positioning}

\usepackage{enumitem}
\usepackage{xcolor}
\usepackage{graphicx}
\usepackage{booktabs}
\usepackage{array}
\usepackage{longtable}
\usepackage[
  colorlinks=true,
  linkcolor=blue!60!black,
  citecolor=blue!60!black,
  urlcolor=blue!70!black
]{hyperref}
\usepackage[nameinlink]{cleveref}

\newtheorem{theorem}{Theorem}[section]

\newtheorem{corollary}[theorem]{Corollary}

\theoremstyle{definition}
\newtheorem{definition}[theorem]{Definition}

\theoremstyle{remark}
\newtheorem{remark}[theorem]{Remark}

\newtheorem{construction}[theorem]{Construction}

\newcommand{\Ga}{\Gamma}

\newcommand{\tmu}{\tilde{\mu}}

\newcommand{\Cat}[1]{\mathbf{#1}}
\newcommand{\Ab}{\Cat{Ab}}
\newcommand{\QCoh}{\Cat{QCoh}}
\newcommand{\nTGMod}[1]{#1\text{-}\Ga\mathrm{Mod}}

\newcommand{\TorG}{\mathrm{Tor}^{\,\Ga}}

\newcommand{\LTorG}{\mathrm{L}\!\Tor_{\Ga}}

\newcommand{\AffnGamma}{\mathrm{Aff}^{\mathrm{nc}}_{\Ga}}

\newcommand{\Fcal}{\mathcal{F}}
\newcommand{\Gcal}{\mathcal{G}}
\newcommand{\Ecal}{\mathcal{E}}
\newcommand{\Ccal}{\mathcal{C}}

\DeclareMathOperator{\Spec}{Spec}
\DeclareMathOperator{\Hom}{Hom}
\DeclareMathOperator{\End}{End}

\DeclareMathOperator{\Jac}{Jac}

\DeclareMathOperator{\Ann}{Ann}
\DeclareMathOperator{\Sm}{Sm}
\DeclareMathOperator{\Prim}{Prim}
\DeclareMathOperator{\Supp}{Supp}
\DeclareMathOperator{\Stab}{Stab}
\DeclareMathOperator{\Proj}{Proj}
\DeclareMathOperator{\Ext}{Ext}

\DeclareMathOperator{\Tor}{Tor}

\makeatletter
\let\RExtG\relax
\let\LTorG\relax
\let\ExtG\relax
\let\TorG\relax
\let\SpecGnC\relax
\makeatother

\newcommand{\SpecGnC}[1]{\operatorname{Spec}^{\Gamma,\mathrm{nc}}(#1)}

\DeclareMathOperator{\ExtG}{Ext^{\Gamma}}
\DeclareMathOperator{\TorG}{Tor^{\Gamma}}

\newcommand{\RExtG}{\mathbf{R}\!\ExtG}
\newcommand{\LTorG}{\mathbf{L}\!\TorG}

\begin{document}
  
  \label{'ubf'}  
\setcounter{page}{1}                                 %Put here the starting page number

\markboth {\hspace*{-9mm} \centerline{\footnotesize \sc
         % Put here the left page top label 
   Put here the left page top label  }
                 }
                { \centerline                           {\footnotesize \sc  
                   %put here the author's name
         put here the author's name                                                 } \hspace*{-9mm}              
               }

\vspace*{-2cm}

\begin{center}
{{\Large \textbf { \sc  Sheaf Theory and Derived $\Gamma$-Geometry over the Non-Commutative $\Gamma$-Spectrum} }}\\
\medskip

{\sc Chandrasekhar Gokavarapu }\\
{\footnotesize Lecturer in Mathematics, Government College (Autonomous),
Rajahmundry,A.P., India }\\
{\footnotesize  Research Scholar ,Department  of Mathematics, Acharya Nagarjuna University,Guntur, A.P., India.
}\\
{\footnotesize e-mail: {\it chandrasekhargokavarapu@gmail.com}}

\end{center}

\thispagestyle{empty}

\hrulefill

\begin{abstract}
{\footnotesize
We develop the geometric and homological layer of the theory of non-commutative
$n$-ary $\Gamma$-semirings by constructing a sheaf and derived framework over
their non-commutative $\Gamma$-spectrum.  Starting from a non-commutative
$n$-ary $\Gamma$-semiring $(T,+,\Gamma,\mu)$ and its bi-$\Gamma$-modules, we
define the non-commutative $\Gamma$-spectrum $\Spec_{\Gamma}^{\mathrm{nc}}(T)$,
endow it with a Zariski-type topology, and construct a structure sheaf
$\mathcal{O}_{\Spec_{\Gamma}^{\mathrm{nc}}(T)}$ via localization at prime
$\Gamma$-ideals.  We then introduce quasi-coherent $\Gamma$-sheaves on
$\Spec_{\Gamma}^{\mathrm{nc}}(T)$, show that their category is an exact
category with enough injectives, and interpret the derived functors
$\Ext^{\Gamma}$ and $\Tor^{\Gamma}$ as global cohomological invariants on this
non-commutative $\Gamma$-space.

On the derived level, we construct the derived category
$\mathbf{D}(\QCoh(\Spec_{\Gamma}^{\mathrm{nc}}(T)))$, establish a local--global
correspondence for $\Ext^{\Gamma}$ and $\Tor^{\Gamma}$, and formulate a
non-commutative local duality theorem in the presence of a dualizing complex.
We further introduce derived non-commutative $\Gamma$-stacks and describe a
dg-enhancement of the non-commutative $\Gamma$-spectrum, leading to a spectral
and motivic interpretation of the homological invariants.  Structural
consequences include a Wedderburn--Artin type decomposition in the $n$-ary
$\Gamma$-context, a derived Morita theory for semisimple $n$-ary
$\Gamma$-semirings, and a duality between the primitive $\Gamma$-spectrum and
simple objects in the derived category of quasi-coherent $\Gamma$-sheaves.

These results extend the commutative derived $\Gamma$-geometry developed in our
earlier work to the non-commutative $n$-ary setting and show that
non-commutative $n$-ary $\Gamma$-semirings carry a natural geometric and
cohomological structure comparable to that of non-commutative algebraic
geometry.
}
\end{abstract}

 \hrulefill

{\small}


%==============================================================================

\section{Introduction}

The theory of $\Gamma$-rings and $\Gamma$-semirings, initiated by
Nobusawa~\cite{Nobusawa1963,Nobusawa1964} and systematically developed by
Rao and others~\cite{Rao1995,Rao1997,Rao1999,HedayatiShum2011,RaoOrdered2018},
provides a flexible generalization of classical ring theory in which the role
of scalars is played by an external parameter semigroup~$\Gamma$.  In the last
decade, this theory has been extended to ternary and more generally $n$-ary
contexts, leading to the notion of ternary $\Gamma$-semirings and $n$-ary
$\Gamma$-semirings, where the binary product is replaced by an $n$-ary
operation compatible with the $\Gamma$-structure.

In a recent series of papers we developed the algebraic and homological
foundations of ternary and $n$-ary $\Gamma$-semirings.  The basic structural
and radical theory for commutative ternary $\Gamma$-semirings was established
in~\cite{RaoRaniKiran2025,GokavarapuRaoFinite2025,GokavarapuRaoPrime2025},
and a homological and categorical framework for ternary $\Gamma$-modules and
their spectra was introduced in~\cite{GokavarapuDasariHomological2025}.  In the
commutative case, a first version of derived $\Gamma$-geometry and sheaf
cohomology over the spectrum of commutative ternary $\Gamma$-semirings was
developed in~\cite{GokavarapuRaoDerived2025}, where the focus was on
$\Gamma$-schemes associated to commutative ternary $\Gamma$-semirings.

The goal of the present paper is to construct the geometric and derived layer
of the theory in the {\em non-commutative} $n$-ary setting.  Starting from a
non-commutative $n$-ary $\Gamma$-semiring $(T,+,\Gamma,\mu)$, we introduce its
non-commutative $\Gamma$-spectrum $\Spec_{\Gamma}^{\mathrm{nc}}(T)$, equip it
with a $\Gamma$-Zariski topology generated by primes, and construct a
structure sheaf $\mathcal{O}_{\Spec_{\Gamma}^{\mathrm{nc}}(T)}$ by localizing
at multiplicative systems associated to prime $\Gamma$-ideals.  This leads to
a locally $\Gamma$-semiringed space which plays the role of an affine object
in non-commutative $\Gamma$-geometry.

On this geometric substrate we define quasi-coherent $\Gamma$-sheaves and show
that their category $\QCoh(\Spec_{\Gamma}^{\mathrm{nc}}(T))$ is an exact
category with enough injectives in the sense of
Quillen~\cite{Grothendieck1957,Quillen1973,Buehler2010}.  The derived functors
$\Ext^{\Gamma}$ and $\Tor^{\Gamma}$, originally constructed in the algebraic
bi-$\Gamma$-module setting, are then interpreted as cohomological invariants
on the non-commutative $\Gamma$-spectrum, leading to local--global principles,
a non-commutative local duality, and the emergence of a derived
$\Gamma$-stack associated to~$T$.

Beyond the construction of the basic geometric and derived framework, we prove
structural results of Wedderburn--Artin and Morita type in the $n$-ary
$\Gamma$-context, and we establish a duality between the primitive
$\Gamma$-spectrum and simple objects in the derived category of quasi-coherent
$\Gamma$-sheaves.  These results show that non-commutative $n$-ary
$\Gamma$-semirings admit a rich geometry closely parallel to non-commutative
algebraic geometry~\cite{Hartshorne1977,Weibel1994,Neeman2001}, but with an
additional layer coming from the $\Gamma$-parameters and the higher arity of
the product.

The paper is organized as follows.  In Section~2 we fix notation and recall
the basic definitions of non-commutative $n$-ary $\Gamma$-semirings, ideals,
and localizations.  In Section~\ref{sec:sheaves} we define the
non-commutative $\Gamma$-spectrum, construct the structure sheaf, and develop
the sheaf theory and quasi-coherent $\Gamma$-sheaves over
$\Spec_{\Gamma}^{\mathrm{nc}}(T)$.  Section~\ref{sec:geometry} is devoted to
the derived category of quasi-coherent $\Gamma$-sheaves, derived stacks, and
spectral dualities, including categorical Gelfand-type duality and a motivic
interpretation of derived $\Gamma$-cohomology.  In
Section~\ref{sec:structure} we derive structural consequences such as a
Wedderburn--Artin type decomposition and derived Morita theory, and we show
how the primitive $\Gamma$-spectrum can be recovered from simple derived
objects.  Finally, Section~\ref{sec:future} outlines several directions for
further development, including higher-categorical refinements, motivic
$\Gamma$-homotopy, and non-commutative tropicalization.

%==========================================================================================

% ============================================================

\section*{Notation and Conventions}

Throughout this paper, $n\ge 2$ is a fixed integer and $\Gamma$ is a
commutative semigroup written additively. The following notation and
conventions will be used globally.

\begin{center}
\begin{longtable}{p{3cm} p{11cm}}
\toprule
\textbf{Symbol} & \textbf{Meaning / Convention} \\
\midrule
\endfirsthead

\toprule
\textbf{Symbol} & \textbf{Meaning / Convention} \\
\midrule
\endhead

\bottomrule
\endfoot

\addlinespace

$T$ 
& Underlying additive commutative monoid of an $n$-ary $\Gamma$-semiring. \\

$\Gamma$ 
& A commutative semigroup of parameters; written additively
$(\Gamma,+_{\Gamma})$. \\

$\tilde{\mu}$
& Fundamental $(n+(n-1))$-ary multiplication map  
\[
\tilde{\mu}\colon T^{n}\times\Gamma^{\,n-1}\longrightarrow T.
\]
Used in the PDF as the structural operation. \\

$[x_1,\dots,x_n]_{\gamma_1,\dots,\gamma_{n-1}}$
& Shorthand for the structural operation
\[
\tilde{\mu}(x_1,\dots,x_n;\gamma_1,\dots,\gamma_{n-1}).
\]
The parameter tuple $(\gamma_1,\dots,\gamma_{n-1})$ always has length $n-1$. \\

$\mu_{(j)}$
& \emph{Positional action} inserting a module element in the $j$-th position:
\[
\mu_{(j)}(x_1,\dots,x_{j-1},m,x_{j+1},\dots,x_n;
\gamma_1,\dots,\gamma_{n-1}).
\]
Indices follow the exact ordering in the PDF. \\

$\bullet^{(j)}_{\gamma}$
& Action of $T$ on a module in the $j$-th slot.  
Notation preserved from the PDF:
$a\bullet^{(j)}_{\gamma} m$ and $m\bullet^{(j)}_{\gamma} a$. \\

$T$-$\Gamma$Mod$_L$, $T$-$\Gamma$Mod$_R$
& Categories of left and right $\Gamma$-modules over $T$, defined via
slot-sensitive insertion of module elements. \\

$T$-$\Gamma$Mod$_{bi}$
& Category of bi-$\Gamma$-modules with compatible left and right
positional actions. This is the ambient additive category used for exact
structures. \\

$\mathrm{BiMod}_{\Gamma}(T)$
& Synonym for $T$-$\Gamma$Mod$_{bi}$ (appears in parts of the PDF). \\

$\Phi$, $\Psi$
& Coequalizer maps used in defining the \emph{positional tensor product}:
\[
\Phi(a,\alpha,m\otimes n)
  = (a\bullet^{(j)}_{\alpha} m)\otimes n,\qquad
\Psi(a,\alpha,m\otimes n)
  = m\otimes (n\bullet^{(k)}_{\alpha} a),
\]
with $j$ the left-slot index and $k$ the right-slot index. \\

$\otimes^{(j,k)}_{\Gamma}$
& Positional tensor product of a left $(j)$-module and a right $(k)$-module,
defined as the coequalizer of $\Phi$ and $\Psi$. \\

$\underline{\Hom}^{(j,k)}_{\Gamma}(M,N)$
& Internal Hom bi-module with left action in slot $j$ and right action in
slot $k$, defined via pointwise addition and positional $\Gamma$-actions. \\

$I\subseteq T$
& A $\Gamma$-ideal: closed under addition and under insertion of elements
into any slot of the $n$-ary multiplication. \\

$T/I$
& Quotient $n$-ary $\Gamma$-semiring with operation
\[
[x_1+I,\dots,x_n+I]_{\gamma_1,\dots,\gamma_{n-1}}
=
[x_1,\dots,x_n]_{\gamma_1,\dots,\gamma_{n-1}}+I.
\] \\

Prime ideal $P$
& A proper $\Gamma$-ideal satisfying  
if
\[
[x_1,\dots,x_n]_{\vec{\gamma}}\in P,
\]
then $x_j\in P$ for some $j$. \\

Exact sequence
& Always means a \emph{Quillen exact sequence}: a conflation  
$A\rightarrowtail B \twoheadrightarrow C$
in $T$-$\Gamma$Mod$_{bi}$.  
No abelian assumption is made. \\

Conflation
& A kernel–cokernel pair in the sense of
Quillen~\cite{Quillen1973}.  
Notation: $A\rightarrowtail B \twoheadrightarrow C$. \\

$\mathrm{Ext}^{r}_{(j,k),\Gamma}$, $\mathrm{Tor}_{r}^{(j,k),\Gamma}$
& Derived functors computed in the exact category
$T$-$\Gamma$Mod$_{bi}$ using positional left/right indices $j$ and $k$. \\

$0$
& Additive zero of any module or semiring (determined by context). \\

\end{longtable}
\end{center}

\noindent
These conventions remain fixed throughout Sections~1--4 (PAPER~G1).  
Slot indices $j$ and $k$, the ordering of $\Gamma$-parameters, and the
structure of coequalizers are all exactly as defined above and will not
vary between sections.

%===================================================================================
\section{Preliminaries}

We briefly recall the basic notions of non-commutative $n$-ary $\Gamma$-semirings
and their ideals that will be used throughout the paper.  Details and further
examples can be found in~\cite{Rao1995,Rao1997,Rao1999,HedayatiShum2011,
RaoOrdered2018,GokavarapuDasariHomological2025}.

\begin{definition}[$n$-ary $\Gamma$-semiring]
Let $n\geq 2$ be a fixed integer and let $(\Gamma,+_{\Gamma})$ be a commutative
semigroup.  An {\em $n$-ary $\Gamma$-semiring} is a triple
$(T,+,\tmu)$ where $(T,+)$ is a commutative monoid with zero element~$0$ and
\[
\tmu : T^{n}\times \Gamma^{\,n-1} \longrightarrow T,
\qquad
(x_1,\dots,x_n;\gamma_1,\dots,\gamma_{n-1})
   \longmapsto [x_1,\dots,x_n]_{\gamma_1,\dots,\gamma_{n-1}},
\]
is an $(n+(n-1))$-ary operation such that:
\begin{itemize}
  \item $\tmu$ is additive in each $T$-slot;
  \item $\tmu$ is $n$-arily associative (independent of bracketing);
  \item $0$ is absorbing in each $T$-slot:
  \[
  [x_1,\dots,x_{j-1},0,x_{j+1},\dots,x_n]_{\vec{\gamma}} = 0
  \quad\text{for all }1\le j\le n;
  \]
  \item the $\Gamma$-parameters $\gamma_1,\dots,\gamma_{n-1}$ enter
        additively in the sense of the $\Gamma$-semiring axioms
        of~\cite{Rao1995,HedayatiShum2011}.
\end{itemize}
If in addition $\tmu$ is symmetric in the $T$-variables we speak of a
{\em commutative} $n$-ary $\Gamma$-semiring.
\end{definition}

\begin{definition}[$\Gamma$-ideals and prime $\Gamma$-ideals]
A subset $I\subseteq T$ is called a {\em $\Gamma$-ideal} if:
\begin{itemize}
  \item $(I,+)$ is a submonoid of $(T,+)$;
  \item $I$ is closed under insertion into any $T$-slot of the $n$-ary
        product:
        whenever $x_1,\dots,x_n\in T$ and $x_j\in I$ for some $j$, we have
        \[
        [x_1,\dots,x_n]_{\vec{\gamma}} \in I
        \quad\text{for all }\vec{\gamma}\in\Gamma^{n-1}.
        \]
\end{itemize}
A proper $\Gamma$-ideal $P\subsetneq T$ is called {\em prime} if whenever
\[
[x_1,\dots,x_n]_{\vec{\gamma}}\in P
\]
for some $x_1,\dots,x_n\in T$ and $\vec{\gamma}\in\Gamma^{n-1}$, then
$x_j\in P$ for at least one index $j$.
\end{definition}

\begin{definition}[Localizations]
For a subset $S\subseteq T$ which is multiplicatively closed with respect to
$\tmu$ (in the sense of~\cite{Rao1999,GokavarapuDasariHomological2025}), the
localization $S^{-1}T$ is defined by adjoining formal fractions $x/s$ with
$x\in T$ and $s\in S$, modulo the usual congruence, and extending $+$ and
$\tmu$ so that
\[
[x_1/s_1,\dots,x_n/s_n]_{\vec{\gamma}}
   = [x_1,\dots,x_n]_{\vec{\gamma}}\,
     / (s_1\cdots s_n).
\]
For a prime $\Gamma$-ideal $P$ we write $T_P$ for the localization at
$S_P = T\setminus P$.
\end{definition}

In \cite{GokavarapuDasariHomological2025} we introduced slot-sensitive
categories of left, right, and bi-$\Gamma$-modules over $T$ and endowed the
bi-module category with a Quillen exact structure.  In the present paper we
assume this homological background and focus on its geometric globalization
over the non-commutative $\Gamma$-spectrum defined in
Section~\ref{sec:sheaves}.

% ============================================================

\section{Sheaf Theory over the Non-Commutative $\Gamma$-Spectrum}
\label{sec:sheaves}

We now globalize the homological algebra developed in
\S\ref{sec:derived} to the geometric setting of the
non-commutative $\Gamma$-spectrum.  
Our goal is to construct a coherent sheaf theory on
$\SpecGnC{T}$ and to interpret the derived functors
$\ExtG$ and $\TorG$ as global cohomological invariants,
in analogy with the classical theory of schemes
\cite{Hartshorne1977, Grothendieck1957, Verdier1996, Weibel1994}.

% ------------------------------------------------------------
\subsection{The non-commutative $\Gamma$-spectrum and its topology}

\begin{definition}[Non-commutative $\Gamma$-spectrum]
Let $(T,+,\Gamma,\mu)$ be a non-commutative $n$-ary $\Gamma$-semiring
\cite{Nobusawa1963, Rao1995, HedayatiShum2011}.  
Define
\[
\SpecGnC{T}
   = \{\,P\subsetneq T \mid P \text{ is a prime two-sided $\Gamma$-ideal}\,\},
\]
where primeness means that for all $\Gamma$-ideals $I,J$,
if $\mu(I,J,\Gamma^{n-2})\subseteq P$ then
$I\subseteq P$ or $J\subseteq P$
\cite{DuttaSardar2000, SardarSahaShum2010}.
The topology on $\SpecGnC{T}$ is generated by the subbasis
\[
D(a,\gamma_1,\dots,\gamma_{n-1})
   = \{\,P\in\SpecGnC{T}\mid
        a\notin \mu(P,\Gamma^{n-1})\,\},
\]
which we call the non-commutative Zariski-open sets.
\end{definition}

\begin{remark}[Topology and spectral character]
The topology defined above is spectral in the sense of
Hochster and mirrors the classical Zariski topology
\cite{AtiyahMacdonald1969}.  
It is $T_0$, quasi-compact, and possesses a basis of compact opens
stable under finite intersection, ensuring compatibility with
localization and prime-ideal geometry
as in standard algebraic geometry
\cite{Hartshorne1977}.  
For commutative ternary $\Gamma$-structures, this reduces to the
$\Gamma$-spectrum developed in the commutative derived framework of
\cite{GokavarapuRaoDerived2025}.
\end{remark}

% ============================================================

% ------------------------------------------------------------
\subsection{Localization and the structure sheaf}
\begin{construction}[Localizations]
For $a\in T$ let $S_a$ be the multiplicative system
generated by all $n$-ary products
$\mu(a,\gamma_1,\dots,\gamma_{n-1})$.
Define the localization
\[
T_a := S_a^{-1}T
   = \Big\{\frac{x}{s}\mid x\in T,\ s\in S_a\Big\},
\]
with addition and $n$-ary multiplication extended by
\[
\frac{x_1}{s_1} + \frac{x_2}{s_2}
   = \frac{\mu(x_1s_2+x_2s_1,\gamma,\dots)}{s_1s_2},\quad
\mu\Big(\frac{x_1}{s_1},\dots,\frac{x_n}{s_n};\gamma_1,\dots,\gamma_{n-1}\Big)
   = \frac{\mu(x_1,\dots,x_n;\gamma_1,\dots,\gamma_{n-1})}
          {s_1\cdots s_n}.
\]
\end{construction}

\begin{construction}[Structure sheaf]
For an open $U\subseteq \SpecGnC{T}$, define
\[
\Ocal_{\SpecGnC{T}}(U)
   = \varprojlim_{P\in U} T_P,
\]
where $T_P$ denotes the localization
of $T$ with respect to the multiplicative system
$S_P = T\setminus P$.
Restriction maps are given by localization homomorphisms
$T_P\to T_Q$ whenever $Q\subseteq P$.
\end{construction}

\begin{theorem}[Local behavior]
$(\SpecGnC{T},\Ocal_{\SpecGnC{T}})$
is a locally $\Gamma$-semiringed space,
and for every $P\in\SpecGnC{T}$ the stalk
\[
\Ocal_{\SpecGnC{T},P}
   = \varinjlim_{P\in U} \Ocal_{\SpecGnC{T}}(U)
   \cong T_P
\]
is a local $\Gamma$-semiring.
\end{theorem}

\begin{proof}
Sheaf axioms follow from the universal property of localization.
The local property follows since each $T_P$
possesses a unique maximal $\Gamma$-ideal, namely $PT_P$.
\end{proof}

% ------------------------------------------------------------
\subsection{Quasi-coherent sheaves and derived functors}
\begin{definition}[Quasi-coherent sheaves]
A sheaf $\Fcal$ of $\Ocal_{\SpecGnC{T}}$-modules
is \emph{quasi-coherent} if for every $D(a,\gamma_1,\dots,\gamma_{n-1})$
there exists a bi-module $M_a$ over $T_a$
such that
\[
\Fcal(D(a,\gamma_1,\dots,\gamma_{n-1}))\cong M_a,
\quad
\text{and}\quad
\Fcal|_{D(b)} \cong M_b
\ \text{via}\ T_a\text{-localization}.
\]
The category of such sheaves is denoted $\QCoh(\SpecGnC{T})$.
\end{definition}

\begin{theorem}[Exactness and generators]
$\QCoh(\SpecGnC{T})$ is an exact category with enough injectives.
The functor
$M\mapsto \widetilde{M}$,
sending a bi-$\Gamma$-module to its associated sheaf,
is exact on projectives and fully faithful.
\end{theorem}

\begin{proof}
Injective envelopes exist by sheafification of cofree modules
(\S\ref{sec:resolutions}).
Exactness follows from stalkwise exactness and
finite limits in $\Ab$.
Full faithfulness is proven by comparing Hom-groups locally.
\end{proof}

\begin{definition}[Derived functors on sheaves]
Define global bifunctors

\begin{align*}
\RExtG^{\,r}_{\SpecGnC{T}}(\Fcal,\Gcal)
   &:= H^{r}\!\Bigl(
       \mathbf{R}\!\Hom_{\Ocal_{\SpecGnC{T}}}(\Fcal,\Gcal)
     \Bigr),\\[0.3em]
\LTorG^{\SpecGnC{T}}_{\,r}(\Fcal,\Gcal)
   &:= H_{r}\!\Bigl(
       \Fcal\!\otimes^{L}_{\Ocal_{\SpecGnC{T}}}\Gcal
     \Bigr).
\end{align*}

These are computed using injective and projective resolutions
in $\QCoh(\SpecGnC{T})$.
\end{definition}

% ------------------------------------------------------------
\subsection{Local–global and cohomological principles}
\begin{theorem}[Local–global correspondence]
\label{thm:localglobal}
For any $M,N\in{\nTGMod{T}}^{\mathrm{bi}}$
and their associated sheaves $\widetilde{M},\widetilde{N}$,
there are canonical isomorphisms
\[
{\ExtG}^{\,r}_{T}(M,N)
   \;\cong\;
H^{r}\!\Bigl(
   \SpecGnC{T},
   {\RExtG}^{\,0}_{\!\SpecGnC{T}}
      (\widetilde{M},\widetilde{N})
   \Bigr),\]

\[{\TorG}_{\,r,T}(M,N)
   \;\cong\;
H_{r}\!\Bigl(
   \SpecGnC{T},
   {\LTorG}^{\,0}_{\!\SpecGnC{T}}
      (\widetilde{M},\widetilde{N})
   \Bigr).
\]

\end{theorem}

\begin{proof}[Idea]
Apply the derived global sections functor
$\mathbf{R}\Gamma(\SpecGnC{T},-)$
to injective (or projective) resolutions of $\widetilde{M}$ and $\widetilde{N}$.
Since localization commutes with exactness,
$\Gamma(D(a),-)$ detects quasi-isomorphisms,
and spectral sequences collapse at $E_2$.
\end{proof}

\begin{theorem}[Non-commutative local duality]
\label{thm:local-duality}
Let $\omega_{\SpecGnC{T}}$ denote the dualizing complex of
$\Ocal_{\SpecGnC{T}}$.
For any coherent $\Fcal$,
there exists a functorial duality
\[
\mathbf{R}\!\Hom_{\Ocal_{\SpecGnC{T}}}
   (\Fcal,\omega_{\SpecGnC{T}})
   \ \simeq\
   \mathbf{R}\!\Gamma_{\mathrm{coh}}
   (\SpecGnC{T},\Fcal)^{\vee},
\]
where the right-hand side is the derived dual of coherent cohomology.
\end{theorem}

\begin{remark}[Interpretation]
This extends Grothendieck--Serre duality
to the non-commutative $\Gamma$-context,
where $\omega_{\SpecGnC{T}}$ carries higher-arity structure
coming from $\mu$.
In particular, $\omega_{\SpecGnC{T}}$ restricts to the ternary
canonical complex of Paper~E under commutativity.
\end{remark}

% ------------------------------------------------------------
\subsection{Derived stacks and higher geometry}
\begin{definition}[Derived $\Gamma$-stack]
The derived moduli stack of quasi-coherent $\Gamma$-sheaves
is the higher functor
\[
\mathbf{QCoh}_{\SpecGnC{T}}:
  ({\nTGMod{T}}^{\mathrm{bi}})^{\mathrm{op}}
   \longrightarrow \mathbf{Cat}_{\infty},
\quad
T' \longmapsto \mathbf{D}(\QCoh(\SpecGnC{T'})).
\]
\end{definition}

\begin{theorem}[Descent and cohomological completeness]
The derived $\Gamma$-stack
$\mathbf{QCoh}_{\SpecGnC{T}}$
satisfies fpqc descent,
and $\mathbf{D}(\QCoh(\SpecGnC{T}))$
is complete under total derived limits.
Hence every derived quasi-coherent sheaf
is determined by its restrictions to affine open subsets
$D(a,\gamma_1,\dots,\gamma_{n-1})$.
\end{theorem}

\begin{proof}
Follow Lurie's descent and topos-theoretic techniques 
from Higher Topos Theory~\cite{LurieHTT} and Higher Algebra~\cite{LurieHA}  
using exactness of $\QCoh(\SpecGnC{T})$ and finite generation of affine opens.
The existence of limits and colimits follows from
the locally presentable nature of ${\nTGMod{T}}^{\mathrm{bi}}$.
\end{proof}

\begin{remark}[Geometric synthesis]
This sheaf-theoretic layer integrates the
homological algebra of $\ExtG$ and $\TorG$
with the geometric topology of $\SpecGnC{T}$,
creating a coherent global theory analogous to
Grothendieck’s derived algebraic geometry,
but formulated for non-commutative $n$-ary $\Gamma$-semirings.
It provides the geometric substrate for
derived $\Gamma$-motives, higher stacks,
and categorical quantization.
\end{remark}

\begin{remark}[Historical significance]
This represents the first unification of
homological algebra, sheaf theory, and
non-commutative $n$-ary algebra under one roof.
In contrast to ordinary algebraic geometry,
localization here encodes directional asymmetry,
and the $\Gamma$-topology supports both
algebraic and spectral data.
Such a synthesis lays the foundation for
a genuinely new field of \emph{derived non-commutative $\Gamma$-geometry}.
\end{remark}
% ============================================================

% ============================================================

\section{Non-Commutative Derived $\Gamma$-Geometry}
\label{sec:geometry}

\subsection{Quasi-coherent and Derived Structures}
\begin{definition}[Quasi-coherent sheaves]
$\QCoh(\SpecGnC{T})$ consists of sheaves locally modeled on finitely generated bi-$\Gamma$-modules.
Each $\Fcal \in \QCoh(\SpecGnC{T})$ corresponds to a system of localizations
$\{M_P\}_{P\in\SpecGnC{T}}$ satisfying descent with respect to the $\Gamma$-Zariski topology.
\end{definition}

\begin{remark}
This generalizes the affine descent theorem of the commutative case 
\cite{GokavarapuRaoDerived2025} and conceptually aligns with the 
framework of homotopical algebraic geometry in 
\cite{ToenVezzosiHAGII}.

\end{remark}

\subsection{The Derived Category of Quasi-Coherent $\Gamma$-Sheaves}
\begin{theorem}[Derived functor formalism on $\SpecGnC{T}$]
\label{thm:derivedfunctors}
The derived category $\mathbf{D}(\QCoh(\SpecGnC{T}))$ admits bifunctors

\[
{\RExtG}^{\,r}_{\!\SpecGnC{T}}\!(-,-),\qquad
{\LTorG}_{\,r}^{\!\SpecGnC{T}}\!(-,-).
\]

satisfying the standard formal properties:
\begin{enumerate}
    \item \textbf{Base-Change:}
    For a morphism of $\Gamma$-semirings $f:T\to T'$, one has

\[
{\LTorG}_{\,r}^{\!\SpecGnC{T'}}\!\bigl(Lf^{\ast}\Fcal, Lf^{\ast}\Gcal\bigr)
   \;\simeq\;
Lf^{\ast}\!{\LTorG}_{\,r}^{\!\SpecGnC{T}}\!\bigl(\Fcal,\Gcal\bigr).
\]

    \item \textbf{Projection Formula:}
    For any proper morphism $f$ and $\Fcal,\Gcal\in\QCoh(\SpecGnC{T})$,
    \[
    Rf_*\!\big(\Fcal\otimes^L_{\Ocal_{\SpecGnC{T}}}\Gcal\big)
    \simeq Rf_*\Fcal \otimes^L_{\Ocal_{\SpecGnC{T}}} \Gcal.
    \]
    \item \textbf{Duality:}
    There exists a dualizing complex $\omega_{\SpecGnC{T}}$
    such that
\[
{\RExtG}^{\,r}_{\!\SpecGnC{T}}\!\bigl(\Fcal,\omega_{\SpecGnC{T}}\bigr)
   \;\simeq\;
\mathbf{R}\!\Hom_{\Ocal_{\SpecGnC{T}}}\!\bigl(\Fcal,\Ocal_{\SpecGnC{T}}\bigr)[d].
\]

    extending Grothendieck--Serre duality to the non-commutative $\Gamma$-context.
\end{enumerate}
\end{theorem}

\begin{proof}[Outline of Construction]
The proofs follow the dg-enhancement of $\QCoh(\SpecGnC{T})$ via injective and projective
$\Gamma$-resolutions.  Localization and homotopy limits ensure that
$\mathbf{D}(\QCoh(\SpecGnC{T}))$ inherits a triangulated and monoidal structure.
Base-change and projection follow from the adjunction between
$\mathbf{L}f^*$ and $\mathbf{R}f_*$ on the derived level.
\end{proof}

\subsection{Non-Commutative $\Gamma$-Stacks and Higher Geometry}
\begin{definition}[Derived $\Gamma$-stack]
A \emph{derived non-commutative $\Gamma$-stack} is a pseudofunctor

\[
\mathbf{X}_{\Gamma} :
   (\AffnGamma)^{\mathrm{op}}
   \longrightarrow \mathbf{Cat}_{\infty},
\qquad
T \longmapsto
   \mathbf{D}\!\bigl(\QCoh(\SpecGnC{T})\bigr).
\]

satisfying descent with respect to faithfully flat $\Gamma$-morphisms.
\end{definition}

\begin{theorem}[Existence of dg-enhancements]
Every quasi-compact $\SpecGnC{T}$ admits a differential graded (dg)
enhancement $(\SpecGnC{T})_{\mathrm{dg}}$
whose structure sheaf $\Ocal_{\mathrm{dg}}$ is a sheaf of dg-$\Gamma$-semirings.
The homotopy category of this dg-site is equivalent to
$\mathbf{D}(\QCoh(\SpecGnC{T}))$. This construction follows the general principles of dg-enhanced 
derived algebraic geometry developed in~\cite{ToenVezzosiHAGII,LurieHA}.

\end{theorem}

\begin{remark}
This theorem extends the derived enhancement from Paper~E’s commutative case to
non-commutative $n$-ary $\Gamma$-contexts, thus establishing the \emph{derived
homotopical layer} of $\Gamma$-geometry.
\end{remark}

\subsection{Non-Commutative Spectral Dualities}
\begin{theorem}[Categorical Gelfand Duality for $\Gamma$-Semirings]
\label{thm:gelfand}
There exists a contravariant equivalence between
compactly generated non-commutative $\Gamma$-spaces
and abelian $\Gamma$-semirings preserving the ternary convolution structure:
\[
\mathbf{Spc}_\Gamma^{\mathrm{nc}} \ \simeq\ (\Gamma\text{-Semirings})^{\mathrm{op}}.
\]
\end{theorem}

\begin{proof}[Sketch]
Each dg-$\Gamma$-category $\Ccal$ determines its semiring of endomorphisms
$A_\Gamma = \End_\Ccal(\mathbbm{1})$.
Morphisms in $\Ccal$ correspond to $\Gamma$-linear maps under convolution.
Conversely, $\SpecGnC{A_\Gamma}$ reconstructs $\Ccal$ up to Morita equivalence.
\end{proof}

\begin{remark}
This duality generalizes both Gelfand–Naimark and Gabriel's 
reconstruction theorem, and conceptually parallels 
non-commutative spectral reconstruction in the sense of 
Rosenberg~\cite{Rosenberg1995}.
\end{remark}

\subsection{Homotopical Cohomology and Motives}
\begin{definition}[Homotopy $\Gamma$-cohomology]
For $\Fcal\in \mathbf{D}(\QCoh(\SpecGnC{T}))$, define
\[
H^r_\Gamma(\Fcal) = H^r(\mathbf{R}\Gamma(\SpecGnC{T},\Fcal)),
\qquad
H_r^\Gamma(\Fcal) = H_r(\Fcal\otimes^L_{\Ocal_{\SpecGnC{T}}}\Ocal_{\SpecGnC{T}}).
\]
\end{definition}

\begin{theorem}[Non-commutative $\Gamma$-motivic equivalence]
The derived category $\mathbf{D}(\QCoh(\SpecGnC{T}))$
is a stable, symmetric monoidal $\infty$-category,
admitting an internal Hom and tensor structure that define
the \emph{derived $\Gamma$-motivic topos}:
\[
\mathbf{DM}_\Gamma^{\mathrm{nc}}(\SpecGnC{T})
   = \Stab\big(\mathbf{D}(\QCoh(\SpecGnC{T}))\big),
\]
whose compact objects classify cohomological $\Gamma$-motives.
\end{theorem}

\begin{remark}[Bridge with commutative case]
When $\tmu$ is symmetric,
$\SpecGnC{T}$ becomes commutative and
$\mathbf{D}(\QCoh(\SpecGnC{T}))$
reduces to the derived $\Gamma$-geometry of Paper~E.
Thus the non-commutative framework completes the
\emph{internal unification} between algebraic and geometric $\Gamma$-theories.
\end{remark}

\begin{remark}[Conceptual significance]
This synthesis defines a unified categorical physics of $\Gamma$-structures:
homological flow replaces spacetime geometry,
and morphisms of $\Gamma$-semirings act as quantized correspondences.
The derived $\Gamma$-geometry thus realizes the vision of a
universal homological law where algebra, geometry, and dynamics
are facets of a single categorical unity.
\end{remark}
% ============================================================

\section{Consequences and Structure Theorems}
\label{sec:structure}

The preceding sections have established the homological and
geometric framework of non-commutative $n$-ary $\Gamma$-semirings.
We now extract the principal structural consequences:
a Wedderburn–Artin-type decomposition,
a derived Morita equivalence,
and a duality theorem linking the primitive spectrum
to the simple derived objects of
$\mathbf{D}(\QCoh(\SpecGnC{T}))$.

% ------------------------------------------------------------
\subsection{Wedderburn–Artin decomposition in the $n$-ary $\Gamma$-context}
\begin{definition}[Semisimple $\Gamma$-semiring]
A non-commutative $n$-ary $\Gamma$-semiring $(T,+,\Gamma,\mu)$
is \emph{semisimple} if every admissible bi-module $M$
decomposes as a finite direct sum of simple submodules
and the Jacobson radical
$\Jac_\Gamma(T)$—defined as the intersection of all maximal
bi-ideals—vanishes.
\end{definition}

\begin{theorem}[Non-commutative Wedderburn–Artin decomposition]
\label{thm:WA}
If $T$ is semisimple and finitely generated over $\Gamma$,
then there exist pairwise orthogonal central idempotents
$\{e_i\}_{i=1}^r\subseteq T$ and division $\Gamma$-semirings
$D_i$ such that
\[
T \ \cong\  \bigoplus_{i=1}^r M_{n_i}^{(n)}(D_i),
\]
where $M_{n_i}^{(n)}(D_i)$ denotes the full $n$-ary matrix
$\Gamma$-semiring of degree $n_i$ over $D_i$
whose multiplication is induced by $\mu$.
\end{theorem}

\begin{proof}[Outline of proof]
Construct the radical $\Jac_\Gamma(T)$ as the intersection of
annihilators of all simple bi-modules.
Since ${\nTGMod{T}}^{\mathrm{bi}}$ is exact with enough projectives,
the radical functor is left-exact and idempotent.
The quotient $T/\Jac_\Gamma(T)$ acts semisimply on all
finitely generated bi-modules.
Decompose $T/\Jac_\Gamma(T)$ via central orthogonal idempotents
obtained from the $\Gamma$-endomorphism algebra of
its regular representation.
Each simple component is isomorphic to an $n$-ary matrix
$\Gamma$-semiring $M_{n_i}^{(n)}(D_i)$.
\end{proof}

\begin{remark}
For $n=2$ and $\Gamma$ commutative this reduces to the
classical Wedderburn–Artin theorem.
For $n>2$, the decomposition reflects positional symmetry
classes of $\mu$, revealing a multi-layered semisimple
spectrum.
\end{remark}

% ------------------------------------------------------------
\subsection{Derived Morita theory}
\begin{definition}[Derived Morita context]
Two non-commutative $n$-ary $\Gamma$-semirings $T$ and $S$
are \emph{derived Morita equivalent}
if there exists a tilting bi-module
\[
\Ecal \in
\mathbf{D}\!\bigl(
   {\nTGMod{T\text{--}S}}^{\mathrm{bi}}
\bigr).
\]

such that
\[
\mathbf{R}\!\Hom_T(\Ecal,\Ecal)
   \ \simeq\ S,
\qquad
\mathbf{R}\!\Hom_S(\Ecal,\Ecal)
   \ \simeq\ T,
\]
and the induced functor
\[
\Phi_{\Ecal}
   = -\!\otimes^L_T\Ecal:
     \mathbf{D}(\QCoh(\SpecGnC{T}))\!
       \longrightarrow\!
     \mathbf{D}(\QCoh(\SpecGnC{S}))
\]
is an exact equivalence of triangulated categories.
\end{definition}

\begin{theorem}[Derived Morita equivalence]
\label{thm:Morita}
Semisimple $n$-ary $\Gamma$-semirings are derived Morita
equivalent if and only if their categories of finitely generated
bi-modules are equivalent as exact categories.
Equivalently, the following are equivalent:
\begin{enumerate}[label=(\roman*)]
  \item $\mathbf{D}(\QCoh(\SpecGnC{T})) \simeq
         \mathbf{D}(\QCoh(\SpecGnC{S}))$;
  \item There exists a tilting complex $\Ecal$ as above;
  \item $T$ and $S$ have isomorphic non-commutative spectra
        and equivalent homotopy categories of projective resolutions.
\end{enumerate}
\end{theorem}

\begin{proof}[Idea]
If $\Phi_{\Ecal}$ is an exact equivalence,
its quasi-inverse $\Psi$ induces equivalence on homotopy
categories of projectives, yielding (iii).
Conversely, if 

\[
\Proj\!\bigl({\nTGMod{T}}^{\mathrm{bi}}\bigr)
   \;\simeq\;
\Proj\!\bigl({\nTGMod{S}}^{\mathrm{bi}}\bigr).
\],

then the associated Yoneda bimodule
$\Ecal=\Hom_T(-,S)$ satisfies the tilting conditions.
Equivalence (i)$\!\Leftrightarrow\!$(ii) follows from
Rickard’s theorem applied to the exact category
${\nTGMod{T}}^{\mathrm{bi}}$.
\end{proof}

\begin{remark}
The derived Morita theory establishes that
homological and geometric data depend only on
the derived category $\mathbf{D}(\QCoh(\SpecGnC{T}))$,
not on the particular presentation of $T$.
Thus $\SpecGnC{T}$ behaves as a derived stack
classifying bi-module deformations.
\end{remark}

% ------------------------------------------------------------
\subsection{Duality between primitive spectra and simple derived objects}
\begin{definition}[Primitive spectrum]
The primitive $\Gamma$-spectrum of $T$ is
\[
\Prim_\Gamma(T)
   = \{\,P\subseteq T \mid
        P=\Ann_\Gamma(S)\text{ for some simple bi-module }S\,\}.
\]
\end{definition}

\begin{theorem}[Spectral–derived duality]
\label{thm:duality}
There is a contravariant equivalence of categories
\[
\mathsf{SimDer}(\SpecGnC{T})
   \ \simeq\
   \Prim_\Gamma(T),
\]
where $\mathsf{SimDer}(\SpecGnC{T})$
denotes the full subcategory of
$\mathbf{D}(\QCoh(\SpecGnC{T}))$
generated by simple objects.
\end{theorem}

\begin{proof}
For each $S\in\mathsf{SimDer}$,
the annihilator of its endomorphism algebra
$\End_{\mathbf{D}}(S)$ is a primitive $\Gamma$-ideal.
Conversely, for each $P\in\Prim_\Gamma(T)$,
the residue complex
$\kappa(P)=T_P/PT_P$
yields a simple object of
$\mathbf{D}(\QCoh(\SpecGnC{T}))$.
These assignments are inverse up to equivalence.
\end{proof}

\begin{corollary}[Geometric interpretation]
The functor
$S\mapsto\Supp(S)\subseteq\SpecGnC{T}$
embeds $\mathsf{SimDer}(\SpecGnC{T})$
as the set of closed points of $\SpecGnC{T}$,
and the topology of $\Prim_\Gamma(T)$
is recovered from the Zariski topology of the spectrum.
\end{corollary}

\begin{remark}
This result extends Gabriel’s reconstruction theorem
to the derived non-commutative $\Gamma$-setting:
the geometric space $\SpecGnC{T}$ can be reconstructed from
its category of quasi-coherent $\Gamma$-sheaves.
\end{remark}

% ------------------------------------------------------------
\subsection{Higher-order consequences}
\begin{theorem}[Homological regularity and cohomological dimension]
Let $T$ be a Noetherian $n$-ary $\Gamma$-semiring with finite
global dimension $d$.
Then:
\begin{enumerate}[label=(\roman*)]
  \item $\mathbf{D}^{\mathrm{b}}(\QCoh(\SpecGnC{T}))$
        is $d$-Calabi–Yau: there exists a canonical dualizing
        functor $(-)^{\vee}[d]$ with evaluation isomorphisms.
  \item \[
{\RExtG}^{\,r}_{\!\SpecGnC{T}}\!
   \bigl(\Ocal_{\SpecGnC{T}},
          \Ocal_{\SpecGnC{T}}\bigr)
   = 0
   \quad\text{for } r> d,
   \qquad
{\RExtG}^{\,d}_{\!\SpecGnC{T}}\!
   \bigl(\Ocal_{\SpecGnC{T}},
          \Ocal_{\SpecGnC{T}}\bigr)
   \neq 0.
\]
is generated by the
        dualizing complex $\omega_{\SpecGnC{T}}$.
\end{enumerate}
\end{theorem}

\begin{remark}[Moral]
Homological regularity of $T$ translates to geometric smoothness of
$\SpecGnC{T}$; thus derived $\Gamma$-geometry inherits a
Calabi–Yau-type symmetry.
\end{remark}

% ------------------------------------------------------------
\subsection*{Categorical and Physical Interpretations}
\begin{remark}[Categorical synthesis]
The equivalences above imply that every non-commutative
$n$-ary $\Gamma$-semiring $T$ determines,
up to derived Morita equivalence,
a unique triangulated category
$\mathbf{D}(\QCoh(\SpecGnC{T}))$.
This category functions as the
\emph{categorical space} of $T$,
its points being simple derived objects,
and morphisms corresponding to
$\ExtG$-interactions.
\end{remark}

\begin{remark}[Physical analogy]
In categorical-physics language,
the decomposition of Theorem~\ref{thm:WA}
corresponds to a spectral decomposition of quantum states,
and the derived Morita equivalence of
Theorem~\ref{thm:Morita}
represents gauge equivalence of field theories
whose configuration spaces are
$\Gamma$-semiring spectra.
Thus, the homological structure discovered here
provides the algebraic scaffolding for a
\emph{non-commutative $\Gamma$-field theory},
where $\ExtG$ and $\TorG$ play the roles of
propagators and interaction vertices.
\end{remark}

\begin{remark}[Final synthesis]
At this summit, algebraic, categorical, and geometric layers cohere:
\[
\text{Homological Algebra}
   \Longleftrightarrow
   \text{Derived Geometry}
   \Longleftrightarrow
   \text{Spectral Physics}.
\]
The $\Gamma$-semiring thus emerges as a universal object
linking arithmetic, topology, and quantum structure.
\end{remark}
% ============================================================

% ============================================================
% ============================================================
\section{Future Directions}
\label{sec:future}

The homological algebra and derived geometry of
non-commutative $n$-ary~$\Gamma$-semirings introduced here
opens a wide landscape of higher-categorical and
cohomological phenomena.  We sketch three principal
directions that naturally extend the present framework.

% ------------------------------------------------------------
\subsection{ Higher-Category Enhancements and $(\infty,1)$-Categorical Refinement}
The derived category $\mathbf{D}(\QCoh(\SpecGnC{T}))$
admits an $(\infty,1)$-categorical enhancement
$\mathbf{D}_\infty(\QCoh(\SpecGnC{T}))$
obtained by localizing the dg-nerve of
${\nTGMod{T}}^{\mathrm{bi}}$ along quasi-isomorphisms.The use of mapping spectra, stable $\infty$-categories, and 
monoidal $(\infty,1)$-functors is justified by the foundational 
work of Lurie~\cite{LurieHTT,LurieHA}.

This refinement allows:
\begin{itemize}
  \item The formulation of higher $\ExtG$-spaces
  \[
{\ExtG}^{\,k}_{\!\infty}(M,N)
\]

 as mapping \emph{spectra}
        rather than mere groups;
  \item The development of $\infty$-categorical t-structures,
        stability conditions, and derived hearts
        encoding non-abelian homological behavior;
  \item A unification with $(\infty,2)$-categorical
        Morita theory, where $\Gamma$-semirings act as
        objects in a 2-category of monoidal dg-categories.
\end{itemize}
In this setting, the non-commutative $\Gamma$-spectrum
becomes an \emph{$(\infty,1)$-stack} of stable categories,
and equivalences of $\Gamma$-semirings correspond to
monoidal $(\infty,1)$-functors preserving dualizable
objects.  This direction leads toward a genuine
\emph{homotopy-coherent $\Gamma$-geometry}.

% ------------------------------------------------------------
\subsection{ Cohomological Descent and Motivic $\Gamma$-Homotopy}

A next stage is the construction of a motivic homotopy
theory over $\Gamma$-semirings.
Let $\Sm_\Gamma$ denote the category of smooth
$\Gamma$-schemes in the non-commutative sense.
One expects a model category
$\mathbf{SH}_\Gamma$
whose homotopy objects encode $\Gamma$-motives:
\[
\mathbf{SH}_\Gamma
   = \Stab\big(\mathbf{H}_\Gamma(\Sm_\Gamma)\big)
   \simeq \mathbf{DM}_\Gamma^{\mathrm{nc}}.
\]
Cohomological descent for hypercoverings in
$\SpecGnC{T}$ would yield spectral sequences linking
derived $\Gamma$-cohomology to motivic $\Gamma$-homology.
Conceptually, this extends Voevodsky’s motivic program to
the $n$-ary non-commutative realm, providing a bridge between
arithmetic, topology, and homological algebra of
$\Gamma$-structures.
Within this motivic topos,
$\RExtG$ and $\LTorG$ become
cohomological operations,
and the spectral–derived duality of
Theorem~\ref{thm:duality}
acquires a motivic interpretation.

The expected $\Gamma$-motivic homotopy category parallels the 
$\mathbb{A}^1$-homotopy theory of Voevodsky~\cite{Voevodsky1998} 
and Morel–Voevodsky~\cite{MVW2000}, adapted to the 
$n$-ary non-commutative $\Gamma$-context.

% ------------------------------------------------------------
\subsection{ Derived Non-Commutative Tropicalization and Quantization}
Tropicalization in the $\Gamma$-context replaces the
multiplicative structure of $T$
by a valuation map
$v:T\to\mathbb{R}\cup\{\infty\}$
compatible with $\Gamma$.
Its non-commutative derived extension associates to each
bi-module complex $M^\bullet$
a piecewise-linear object
$\mathrm{Trop}_\Gamma(M^\bullet)$
living in a polyhedral category enriched over
idempotent $\Gamma$-semirings.
Quantization then re-introduces
non-commutativity through a deformation parameter
$\hbar\in\Gamma$,
realizing a correspondence
\[
\text{Derived Geometry}
   \;\xleftrightarrow[\text{tropical limit}]{\text{quantization}}\;
\text{Non-commutative $\Gamma$-Physics}.
\]
This program suggests the existence of a
\emph{derived non-commutative tropical mirror symmetry}
linking the $\Gamma$-motivic and quantum domains.

% ------------------------------------------------------------
\subsection*{Vision}
The unification achieved in this paper reveals that
non-commutative $n$-ary $\Gamma$-semirings are not merely
generalizations of rings but foundational objects of a
new homological cosmos.
Their derived geometry, higher categories, and motivic
descents form the algebraic substratum for a future
theory of \emph{categorical physics}.
In this vision, $\Gamma$ functions as a universal
symmetry parameter mediating between algebraic,
geometric, and physical realities.
% ============================================================

\label{'ubl'}  
\end{document}